\UseRawInputEncoding
\documentclass[11pt]{article}
\usepackage{amsmath,amssymb,amsthm,latexsym,amstext}
\usepackage[mathscr]{eucal}
\usepackage{rotate,graphics,epsfig}
\usepackage{float}
\usepackage{color}
\usepackage{cite}
\usepackage{subfigure}
\def\EquationsBySection{\def\theequation
{\thesection.\arabic{equation}}%
\@addtoreset{equation}{section}}

  \newtheorem{defn}{Definition}[section]
\newtheorem{thm}{Theorem}[section]
       
       \newtheorem{lem}{Lemma}[section]

\newcommand\old[1]{}
\makeatother

\EquationsBySection

\begin{document}
\date{}
\pagestyle{plain}
 \title{Random attractors of a stochastic Hopfield neural network model  with  delays}

\author{ Wenjie Hu$^{1,2}$, Quanxin Zhu$^{1}$ and  Peter E. Kloeden$^{3}$ \\
 $^1$\textit{{\footnotesize MOE-LCSM, School of Mathematics and Statistics, Hunan Normal University, Changsha 410081, China;}}\\
   $^2$\textit{{\footnotesize
 Journal House, Hunan Normal University, Changsha, Hunan 410081, China,}}\\
$^3$\textit{{\footnotesize Mathematisches Institut, Universit\"{a}t T\"{u}bingen, 72076 T\"{u}bingen, Germany}
     }}
\maketitle
\begin{abstract}
The  global asymptotic behavior of a stochastic  Hopfield neural network model (HNNM) with delays  is explored by studying the existence and structure of random attractors. It is first proved that the trajectory field of the stochastic delayed HNNM admits an almost sure continuous version, which is compact for $t>\tau$ (where $\tau$ is the delay) by a delicate construction based on the random semiflow generated by the diffusion term.  Then, this version is shown to generate a  random dynamical system (RDS) by piece-wise linear approximation, after which  the  existence of a random absorbing set is obtained by a careful  uniform apriori estimate of the solutions. Subsequently,  the pullback asymptotic compactness of the RDS generated by the stochastic delayed HNNM is proved and hence  the existence of  random attractors is obtained. Moreover,  sufficient conditions under which the attractors turn  out to be an exponential attracting stationary solution are given.  Numerical simulations are also conducted at last to illustrate the effectiveness of the established results.
\end{abstract}


\section{Introduction and examples}
In \cite{hopfield1984neurons}, Hopfield proposed the following prominent HNNM
\begin{equation}\label{1.1}
 \displaystyle\frac{d u_i(t)}{d
 t}=-c_i u_i(t)+\sum_{j=1}^{n} h_{i j} f_{j}\left(u_{j}\left(t\right)\right).
\end{equation}
Here, $u_i(t), i=1,2, \cdots, n$  represent  the state of the $i$th neuron. $c_i$ stands for the reset velocity of the $i$th neuron. $\mathbf{H}=(h_{ij})$ corresponds to the connection matrix between neurons at time $t$, where $h_{ij}, i, j=1,2, \cdots, n$ are the weight coefficients measuring  the connection strength between different neurons. If the output of the $j$th neuron excites the neuron $i$, then we have $h_{ij}\geq 0$, otherwise,  $h_{ij} = 0$.  $f_i: R\rightarrow R, i=1,2, \cdots, n$ are the neuron activation functions. Due to the significant roles that  the HNNM \ref{1.1} has played in the areas of signal processing, optimization, parallel computations, pattern recognition, associate memories and so forth, it has  been extensively and  intensively investigated including both its dynamics analysis and real world applications from the time  that it was proposed. See, for instance \cite{Hopfield1986Computing}, \cite{hopfield1984neurons}, \cite{tank1987neural} and the references therein.

In model \ref{1.1}, it is assumed that   the updating and propagation of neuron signals occur instantaneously  without  time delay. Nevertheless,  time delays are omnipresent both in neural processing and in signal transmission and  must be  taken into
account in  the  communication channels. Hence, Marcus and Westervelt \cite{marcus1989stability} incorporated time delays arising from distance between the neurons or axonal conduction time between neurons into system \ref{1.1} leading to the following delayed neural network model
\begin{equation}\label{1.2}
 \displaystyle\frac{d u_i(t)}{d
 t}=-c_i u_i(t)+\sum_{j=1}^{s} h_{i j} f_{j}\left(u_{j}\left(t\right)\right)+\\
\sum_{j=1}^{n} b_{i j} g_{j}\left(u_{j}\left(t-\tau_{j}\right)\right),
\end{equation}
where  $\tau_{ j},j=1,\cdots,n$ denote  the delay caused by  the transmission  along the axon of the $j$-th unit.  $b_{i j}$ represents the strength of the $j$-th unit on the $i$-th unit at time  $t-\tau_{j}$.  The $g_{j}$ term gives  the corresponding output of the $j$-th unit at time $t-\tau_{j}$.  The dynamical behavior of solutions to \ref{1.2} has drawn  much research interest pertaining to both dynamics and applications. For example,
\cite{Morita1993Associative} and \cite{YOSHIZAWA1993Capacity} compared the memory capacity of \ref{1.2} under different activation functions, while \cite{Gopalsamy1994Stability} and \cite{Van1998Global} obtained the global asymptotic stability of the steady state by  the Lyapunov method.

In  real nervous systems, signal transmission along axons is a noisy process caused by random fluctuations from the release of neurotransmitters and other random effects, as pointed by \cite{Haykin}. Therefore,  a more realistic and accurate mathematical model should be the following stochastic delayed HNNM
\begin{equation}\label{1.3}
\left\{\begin{array}{l}
d u_{i}(t)=[-c_{i} u_{i}(t)+\sum_{j=1}^{n} h_{i j} f_{j}\left(u_{j}(t)\right)+\\
\sum_{j=1}^{n} b_{i j} g_{j}\left(u_{j}\left(t-\tau_{j}\right)\right)] d t+\sum_{j=1}^{n} \sigma_{i j}u_{j}(t) d w_{j}(t) \\
u_{i}(\xi)=\phi_{i}(\xi), \quad-\max _{1 \leq k \leq n} \tau_{k} \leq \xi \leq 0,
\end{array}\right.
\end{equation}
which is obtained by adding an multiplicative noise $\sum_{j=1}^{n} \sigma_{i j}u_{j}(t) d w_{j}(t)$ to \ref{1.3}. Here, $\sigma_{i j}$ are parameters representing the  density of stochastic effects, $w_j(t), j=1,2,\cdots, n$ are mutually independent two-sided real-valued Wiener process  on $\left(\Omega, \mathcal{F}, P\right)$ with a filtration $\{\mathcal{F}_t, t\geq 0\}$. $\phi_{i}(\xi), i=1, \ldots, n$  are the initial conditions. Let $\tau=\max _{1 \leq k \leq n} \tau_{k}$ and  take
$$u(t)=\left(u_{1}(t), u_{2}(t), \ldots, u_{n}(t)\right)^{\mathrm{T}},$$
$$ u_t(\tau)=\left(u_{1}\left(t-\tau_{1}\right), u_{2}\left(t-\tau_{2}\right), \ldots, u_{n}\left(t-\tau_{n}\right)\right)^{\mathrm{T}},$$
$$\quad C=\operatorname{diag}\left(c_{1}, c_{2}, \ldots, c_{n}\right), \quad H=\left(h_{i j}\right)_{n \times n},  \quad B=\left(b_{i j}\right)_{n \times n}, $$
$$f(u)=\left(f_{1}\left(u_{1}\right), f_{2}\left(u_{2}\right), \ldots, f_{n}\left(u_{n}\right)\right)^{\mathrm{T}},$$
 $$g(u)=\left(g_{1}\left(u_{1}\right), g_{2}\left(u_{2}\right), \ldots, g_{n}\left(u_{n}\right)\right)^{\mathrm{T}},$$
$$\Sigma  =(\sigma_{i j})_{n \times n}, $$
$$w(t)=\left(w_{1}(t), w_{1}(t), \ldots, w_{n}(t)\right)^{\mathrm{T}}, $$
$$\phi(t)=\left(\phi_{1}(t), \phi_{2}(t), \ldots, \phi_{n}(t)\right)^{\mathrm{T}}.$$
Then, model \ref{1.3} can be  rewritten  in the following matrix form
\begin{equation}\label{1.4}
\left\{\begin{array}{l}
d u(t)=\left[-C u(t)+H f(u(t))+B g\left(u(t-\tau)\right)\right] d t\\
+\Sigma u(t)\diamond w(t),\\
u(t)=\phi(t), \quad-\tau \leqslant t \leqslant 0, \quad
\end{array}\right.
\end{equation}
in which $\phi=\{\phi(s):-\tau \leqslant s \leqslant 0\}$ is $\mathcal{C}=C([-\tau, 0],\mathbb{R}^{n})$-valued function, where $\mathcal{C}$ is the
space of all continuous $\mathbb{R}^{n}$-valued functions defined on $[-\tau, 0]$ with a norm $\|\xi\|=\sup \{|\xi(t)|:-\tau \leqslant t \leqslant 0\}$ and $|\cdot|$ is the Euclidean norm of a vector $x \in \mathbb{R}^{n}$. Here $u(t)\!\diamond \!dw(t)$ is interpreted as the vector with $j$th component $u_j(t) dw_j(t)$, $j$ $=$ $1$, $\cdots$, $n$.

The local stability around equilibria of the stochastic HNNM and  the stochastic delayed HNNM  has been widely and thoroughly studied in the past decades by employing It\^{o}'s formula, stochastic analysis techniques and inequality techniques. See, for instance, \cite{SC2007pth,WS2005mean,LD2017mean} and  \cite{ZC2014pth} investigate stability of  \ref{1.4} with fixed or variant delays. On the other hand, there are also some other works concerning global stability for equilibrium of stochastic delayed HNNM by adopting the  Lyapunov-Krasovskii functional methods together with stochastic analysis techniques, see for example \cite{LYW08,WWC09,WWC11} and the references therein.

Nevertheless, Lyapunov-Krasovskii functional methods depend critically on the construction of specified Lyapunov-Krasovskii functional which is different from equation to equation. Moreover, the criterion is in linear matrix inequality form that depends highly on numerical tools,  and the condition on the drift term and the random coefficient is strict. From the dynamical systems point of view, if we can find a compact set that is invariant under the system and attracts all orbits of the system, i.e., if we can find a global attractor of the system, then we can obtain the asymptotic behavior of the system by investigating the structure of the attractor. This approach is effective  especially for infinite dimensional dynamical systems generated by partial differential equations or delay differential equations. Generally, the attractors of dynamical systems may be  fixed points,  periodic orbits, invariant manifolds  or chaotic attractors under certain conditions. Therefore, by investigating the attractors of dynamical systems, one can obtain a full view of complex dynamics of the systems. However, to our   knowledge, the existence and structure of attractors for the stochastic delayed HNNM \ref{1.4} have not been studied, although the existence and structure of attractors of the deterministic delayed  HNNM, the delayed non-autonomous neural network models as well as random neural field lattice models have been studied in \cite{liz2013attractivity,HZ22,SWHK20,HKU19,YWK21,WKH21}.  It should be pointed out that, in the very recent work \cite{SWHK20}, the authors investigated the attractors of a random recurrent neural networks with delays. Unlike the delayed random neural networks that naturally generates an RDS, the existence of an RDS generated by stochastic delayed HNNM \ref{1.4} needs much effort to established. Indeed, the main difficulty for obtaining the attractors of stochastic delayed HNNM \ref{1.4} lies in showing the model generates a random dynamical system,  which  naturally holds for deterministic case.

There are only very few papers tackling the attractors for the stochastic delayed differential  equations, see for example, \cite{WK}, \cite{YLJ2010} and \cite{GL2022}. In \cite{WK}, the authors proved that the stochastic delayed differential  equation  generates a mean square RDS processing  a mean-square random attractor in the mean-square sense \cite{PT2012}, which is not the random attractor in the classical pathwise sense. For a detail concept of the mean-square random attractors and mean-square RDSs, the readers are refereed to the review paper \cite{CK15}. In \cite{YLJ2010} and \cite{GL2022} the authors obtained  the existence and structure of random attractors for  stochastic delayed differential  equations with additive noise.
It should be noted that the exponential Ornstein-Uhlenbeck process transformation as \cite{DLS03} did for stochastic (partial) partial differential equations may also be effective for transforming the stochastic delayed equation into random delayed equation, but  a time delay then occurs in the  noise term, which makes it difficult to establish the existence  of a random dynamical system as well as that of tempered absorbing sets. This motivates us to propose a new method  to investigate  the existence of an RDS generated by \ref{1.4} with general multiplicative noise, the existence and structure of random attractors of \ref{1.4}.

We prove that the trajectory field of the stochastic delayed HNNM admits an almost sure continuous and measurable  compact version  for $t>\tau$  which is constructed by random semiflow generated by the linear multiplicative noise.  Then, this version is shown to generate a  RDS  by the trick of piece-wise linear  approximation of the version. The idea originates from the early works \cite{M90} \cite{MS2003}, \cite{MS2004} where stochastic delayed differential equations are recast  into a Hilbert space to study the Lyapunov spectrum and invariant manifolds.  Our work can be regarded as an  attempt to study the random attractors of the HNNM \ref{1.4} in Banach spaces. Very recently, \cite{VRS} investigated the  singular stochastic delay differential equations  in Banach spaces by  rough path theory, which is quite different from approach of the present work. Our main contributions are in the following three folds. Firstly, a novel method for proving the stochastic delayed HNNM  generates an RDS in the natural phase space, i.e. $\mathcal{C}=C([-\tau, 0],\mathbb{R}^{n})$ is established, which is quite different from the situation in the early works \cite{M90} where a Hilbert space is taken as the phase space. Secondly, existence of random attractors of the model under appropriate conditions have been proved. Thirdly, optimal conditions under which that the attractor is a  random stationary solution is obtained.  vMoreover, the method can also be generalized to other stochastic evolution equations with delays, such as the stochastic delayed partial differential equations, for tackling the problem  of topological dimension estimation, Lyapunov exponents,  as well as the invariant manifolds that depends on showing the existence of RDSs.

This paper is organized as follows. In Section 2, we first introduce the theory of RDSs and random attractors as well as some notations and preliminary lemmas needed for the proof of our main results. In Section 3, based on the  random semiflow generatated by the diffusion term of \ref{1.4}, we  transform  \ref{1.4} into a random delayed differential equation  admitting a measurable version that generates an RDS. For the purpose of showing  the existence of random attractor for  \ref{1.4}, we first give a uniform a priori estimate of the solution  and then show the asymptotic compactness of the RDS generated by \ref{1.4} in Section 4, which implies  the existence of random attractor following the results in  \cite{9,10}. In Section 5, we derive  sufficient conditions for  the random attractor being an exponential stationary solution.  Finally, we illustrate the obtained results  by numerical simulations.

\section{Preliminaries}
We first introduce the concepts of a random attractor and an RDS  as well as some important lemmas following \cite{AL,17,9,10} and \cite{FS}.

\begin{defn}\label{defn1}
Let $\left\{\theta_{t}: \Omega \rightarrow \Omega, t \in \mathbb{R}\right\}$ be a family of measure preserving transformations such that $(t, \omega) \mapsto \theta_{t} \omega$ is measurable and $\theta_{0}=\mathrm{id}$, $\theta_{t+s}=\theta_{t} \theta_{s},$ for all $s, t \in \mathbb{R}$. The flow $\theta_{t}$ together with the probability space $\left(\Omega, \mathcal{F}, P,\left(\theta_{t}\right)_{t \in \mathbb{R}}\right)$ is called a metric dynamical system.
\end{defn}

 For a given separable Hilbert space $(X, \|\cdot\|_X)$,  denote by $\mathcal{B}(X)$ the  Borel-algebra of open subsets in $X$.

\begin{defn}\label{defn2}
A mapping $\Phi: \mathbb{R}^{+} \times \Omega \times X   \rightarrow X$ is said to be a random dynamcial system  (RDS) on a complete separable metric space  $(X,d)$ with Borel  $\sigma$-algebra  $\mathcal{B}(X)$ over the metric dynamical system $\left(\Omega, \mathcal{F}, P,\left(\theta_{t}\right)_{t \in \mathbf{R}}\right)$ if \\
(i) $\Phi(\cdot, \cdot, \cdot): \mathbb{R}^{+} \times \Omega \times X   \rightarrow X$ is $(\mathcal{B}(\mathbb{R}^{+})\times \mathcal{F}\times\mathcal{B}(X), \mathcal{B}(X))$-measurable;\\
(ii) $\Phi(0, \omega,\cdot)$ is the identity on  $X$ for $P$-a.e. $\omega \in \Omega$;\\
(iii) $\Phi(t+s, \omega,\cdot)=\Phi(t, \theta_{s} \omega,\cdot) \circ \Phi(s, \omega,\cdot),   \text { for all } t, s \in \mathbb{R}^{+}$ for $P$-a.e. $\omega \in \Omega$.\\
A RDS $\Phi$ is continuous or differentiable if $\Phi(t, \omega,\cdot): X \rightarrow X$ is continuous or differentiable for all $t\in \mathbb{R}^+$ and $P$-a.e. $\omega \in \Omega$.
\end{defn}

\begin{defn}\label{defn3} A set-valued map $\Omega \ni \omega \mapsto D(\omega) \in 2^{X}$ is said to be a random set in $X$ if the mapping $\omega \mapsto d(x, D(\omega))$ is $(\mathcal{F}, \mathcal{B}(\mathbb{R}))$-measurable for any $x \in X,$ where $d(x, D(\omega))\triangleq \inf _{y\in  D(\omega)} \mathrm{d}(x, y)$ is the distance in $X$ between the element $x$ and the set $D(\omega)  \subset X$.
\end{defn}

\begin{defn}\label{defn4} A random set $\{D(\omega)\}_{\omega \in \Omega}$ of $X$ is called tempered with respect to $\left(\theta_{t}\right)_{t \in \mathbb{R}}$ if for $P$-a.e. $\omega \in \Omega$,
$$
\lim _{t \rightarrow \infty} e^{-\beta t} d\left(D\left(\theta_{-t} \omega\right)\right)=0, \quad \text { for all } \beta>0,
$$
where $d(D)=\sup _{x \in D}\|x\|_{X}$.
\end{defn}

\begin{defn}\label{defn5}
Let $\mathcal{D}=\{\{D(\omega)\}_{\omega \in \Omega}: D(\omega)\subset X, \omega\in\Omega\}$ be a family of random sets.  A random set $\{K(\omega)\}_{\omega \in \Omega}  \in \mathcal{D}$ is said to be a $\mathcal{D}$-pullback absorbing set for $\Phi$ if for $P$-a.e. $\omega \in \Omega$ and for every $B \in \mathcal{D},$ there exists $T=T(B, \omega)>0$ such that
\[
\Phi\left(t, \theta_{-t} \omega,B\left(\theta_{-t} \omega\right)\right)  \subseteq K(\omega) \quad \quad \text { for all } t \geq T.
\]
If, in addition, for all  $\omega \in \Omega, K(\omega)$ is a closed nonempty subset of $X$ and $K(\omega)$ is measurable in $\Omega$ with respect to $\mathcal{F},$ then we say $K$ is a closed measurable $\mathcal{D}$-pullback absorbing set for $\Phi$.
\end{defn}

\begin{defn}
 An RDS $\Phi$ is said to be $\mathcal{D}$-pullback asymptotically compact in $X$ if for $P$-a.e. $\omega \in \Omega$, $\left\{\Phi\left(t_{n}, \theta_{-t_{n}} \omega, x_{n}\right)\right\}_{n \geq 1}$ has a convergent subsequence in $X$ whenever $t_{n} \rightarrow \infty$ and $x_{n} \in D\left(\theta_{-t_{n}} \omega\right)$ for any given $D \in \mathcal{D}$.
\end{defn}

\begin{defn}\label{defn7}
A compact random set $\mathcal{A}(\omega)$ is said to be a $\mathcal{D}$-pullback random attractor associated to the RDS
$\Phi$ if it satisfies the invariance property
$$\Phi(t, \omega) \mathcal{A}(\omega)=\mathcal{A}\left(\theta_{t} \omega\right), \quad \text { for all } t \geq 0 $$
and the pullback attracting property
\[
\lim _{t \rightarrow \infty} \operatorname{dist}\left(\Phi\left(t, \theta_{-t} \omega\right)D\left(\theta_{-t} \omega\right), \mathcal{A}(\omega)\right)=0,
\]
 for all $ t \geq 0, D \in \mathcal{D}$ P-a.e. $\omega\in \Omega$.where  $\operatorname{dist} (\cdot, \cdot)$ denotes the Hausdorff semidistance
\[
\operatorname{dist}(A, B)=\sup _{x \in A} \inf _{y\in  B} \mathrm{d}(x, y), \quad A, B \subset  X.
\]
\end{defn}

\begin{lem}\label{lem1}
Let RDS $(\theta, \Phi)$ be a continuous RDS. Suppose that $\Phi$ is   pullback asymptotically compact and has a closed pullback $\mathcal{D}$-absorbing set $K=\{K(\omega)\}_{\omega \in \Omega} \in \mathcal{D}$. Then it possesses a random attractor $\{\mathcal{A}(\omega)\}_{\omega \in \Omega},$ where
$$
\mathcal{A}(\omega)=\cap_{\tau \geq 0} \overline{\cup_{t \geq \tau} \Phi\left(t, \theta_{-t} \omega, K\left(\theta_{-t} \omega\right)\right)}.
$$
\end{lem}
F
The following Gr\"{o}nwall inequality that will be frequently used in our subsequent proofs.
\begin{lem}\label{lem2}
Let $T>0$ and $u, \alpha, f$ and $g$ be non-negative continuous functions defined on $[0, T]$ such that
\[
u(t) \leq \alpha(t)+f(t) \int_{0}^{t} g(r) u(r) d r, \quad \text { for } t \in[0, T].
\]
Then
\[
u(t) \leq \alpha(t)+f(t) \int_{0}^{t} g(r) \alpha(r) e^{\int_{r}^{t} f(\tau) g(\tau) d \tau} d r, \quad \text { for } t \in[0, T].
\]
\end{lem}

\section{Existence of an RDS}\label{S3:SSG}
Throughout the remaining part of this paper, we always make the following assumptions on the nonlinear term $f$ and $g$.

$\mathbf{Hypothesis \  A1}$ The functions $f(x)$ and $g(x)$ are globally Lipschitz continuous with Lipschitz constants $L_f $ and $L_g$.

$\mathbf{Hypothesis \  A2}$ $f(x)$ and $g(x)$ are bounded with $\mathbf{0}$ being a fixed point, that is $f(\mathbf{0})=g(\mathbf{0})=0$ and there exists a $M>0$ such that $f(x)\leq M$ for any $x\in \mathbb{R}^n$ and $g(\phi)\leq M$ for any $\phi\in \mathcal{C}$.

By \cite{M1984} Theorem 2.1 and \cite{GL2022} Lemma 1, we have the following results concerning the existence of a continuous $\mathcal{F}_{t}$-measurable solution to \eqref{1.4}.
\begin{lem}\label{lem2.1}
 Assume that $\mathbf{Hypothesis \  A1-A2}$ hold, then for P-a.e. $\omega\in \Omega$, the stochastic delayed HNNM \eqref{1.4} admits a unique solution $u(\cdot,\omega,\phi) \in C\left([-\tau, \infty), \mathbb{R}^{n}\right)$, which is adapted to $\{\mathcal{F}_t\}_{t\in[0, \infty)}$ for any initial condition  $\phi\in \mathcal{C}$.
\end{lem}

In the following, we are devoted to showing the stochastic delayed HNNM \eqref{1.4} generates a RDS.  For a measurable process $u:[-\tau, \infty)\times\Omega \rightarrow \mathbb{R}^{n}$, the segment $u_{t}: [-\tau,0] \times \Omega \rightarrow \mathbb{R}^{n}, t \in \mathbb{R}^{+}$, is defined by
$$
u_{t}(s):=u(t+s), \quad \quad s \in [-\tau,0].
$$
 Consider the following stochastic differential equation obtained by the diffusion term of  \eqref{1.4}
 \begin{equation}\label{2.1}
\left\{\begin{array}{l}
d v(t)=\Sigma v(t)\!\diamond \!d w(t),\\
v(0)=\mathrm{id}_{\mathbb{R}^{n}},
\end{array}\right.
\end{equation}
where  $\mathrm{id}_{\mathbb{R}^{n}}$ be the identity operator on $\mathbb{R}^{n}$. We assume here and elsewhere that the  components $w_i(t)$ of the vector $w(t)$ are mutually independent standard Wiener processes.

 Let $v(t, \omega, \cdot): \mathbb{R}^{n} \rightarrow \mathbb{R}^{n}$ be the solution of \eqref{2.1} and it is well known from \cite{AL} that \eqref{2.1} admits a linear continuous solution which is bounded for P-a.e. $\omega \in \Omega$. Denote  its bound by $L_v$, i.e. $\|v(t, \omega, \cdot)\|\leq L_v$ for all $t\in \mathbb{R}^n$ and P-a.e. Moreover, it process a bounded continuous  inverse $v^{-1}(t, \omega, \cdot)$.
Hence, the solution $v(t,\omega,\cdot): \mathbb{R}\times \Omega\rightarrow L(\mathbb{R}^{n})$ is $\{\mathcal{F}_t\}_{t\geq0}$-adapted.

Define the following operator
\begin{equation}\label{2.2}
\begin{aligned}
A(t,\omega, \phi)\triangleq &v^{-1}(t, \omega, \cdot)\circ  [-Cv(t,\omega,\phi(0))+H f(v(t,\omega,\phi(0))) \\
&+Bg\left(v_t(\cdot,\omega,\phi)\right)]
\end{aligned}
\end{equation}
with $v_{t}(\cdot, \omega, \cdot): \mathcal{C}\rightarrow \mathcal{C}$ being defined by
\begin{equation}\label{2.3}
\begin{aligned}
v_{t}(s, \omega, \phi)&=\left\{\begin{array}{ll}
v(t+s, \omega, \phi(0)), & t+s\geq 0, \\
\phi(t+s), & -\tau \leq t+s<0.
\end{array}\right.\\
\end{aligned}
\end{equation}
Assume that $g(\varphi)=L\varphi+\tilde{g}(\varphi)$ for any $\varphi \in \mathcal{C}$, where $L$ is the linear part of $g$. (Hence, the Lipschitz constant of $\tilde{g}$ denoted by $L_{\tilde{g}}$ is $L_{g}-\|L\|$). Consider the linear random functional differential equation
\begin{equation}\label{2.7}
\left\{\begin{array}{l}
\frac{d \bar{u}(t)}{d t}=-C\bar{u}(t)+v^{-1}(t, \omega,\cdot)\circ BLv_t(\cdot,\omega,\bar{u}_t)\triangleq \tilde{L}(t,\omega)\bar{u}_t, \\
 \bar{u}_{0}=\phi \in \mathcal{C}
\end{array}\right.
\end{equation}
for any $t>0$.

For the pathwise deterministic functional differential equation \eqref{2.7}, it follows from \cite{JH} that the fundamental solution \eqref{2.7}, denoted by $S^\phi(\cdot,\omega):[0,\infty)\rightarrow  \mathbb{R}^{n^2}$  is defined by
\begin{equation}\label{2.5b}
\begin{aligned}
S^{\phi}(t,\omega) & =(S * \phi)(t,\omega)
  \triangleq S(t,\omega) \phi(0)\\
&+\int_{-\tau}^{0} \int_{\xi}^{0} S(t+\xi-s,\omega) \mathrm{d} \eta(\xi,t,\omega) \phi(s) \mathrm{d} s.
\end{aligned}
\end{equation}
for any $t \geq 0$ and P-a.e. $\omega\in \Omega$. Here, $\eta:[-\tau, 0]\times[0, \infty)\times\Omega \rightarrow \mathbb{R}^{n^{2}}$ is  an $n \times n$ matrix-valued function whose elements are of bounded variation such that
$$
\tilde{L}(t,\xi) \phi=\int_{-\tau}^{0} \mathrm{~d} \eta(\xi,t,\omega) \phi(\xi)
$$
for P-a.e. $\omega\in \Omega$ due to  the Riez representation theorem.

For any fixed $\omega\in \Omega$ and $t>0$, the characteristic equation of \eqref{2.7} is given by
\begin{equation}\label{2.8}
\begin{array}{l}
\operatorname{det}\left[\lambda \operatorname{Id}-\tilde{L}(t,\omega)\left(\mathrm{e}^{-\lambda(\cdot)}\right)\right]=0,
\end{array}
\end{equation}
of which  the spectrum is defined as
$$
\varrho=\max \left\{\operatorname{Re} \lambda: \operatorname{det}\left[\lambda \operatorname{Id}-\tilde{L}(t,\omega)\left(\mathrm{e}^{-\lambda(\cdot)}\right)\right]=0\right\}.
$$
In the reminder of this paper, we always assume that $\varrho<0$ and hence we have the following results by \cite{JH}.
 \begin{lem}\label{lem4}If $\varrho<0$, then there exist positive constants $\gamma$ and $K_{0}$ such that
\begin{equation}\label{2.8a}|(S* \phi)(t,\omega)|<K_{0} \mathrm{e}^{-\gamma t}\|\phi\|\end{equation}
 for all $t \geq 0$, P-a.e. $\omega\in \Omega$ and  any $\phi \in \mathcal{C}$. In particular, there exists a positive constant $K_{1}$ such that
\begin{equation}\label{2.8b}
|S(t,\omega)| \leq K_{1} \mathrm{e}^{\varrho t / 2} \text { for all } t \geq 0, \text { P-a.e.  } \omega\in \Omega.
\end{equation}
\end{lem}
We next show the global existence and properties of solution to a random functional differential equation generated by the operator \eqref{2.2}.
\begin{lem}\label{lem5}
Assume that $\mathbf{Hypothesis \  A1-\ A2}$ hold and $\varrho<0$, then the following random retarded  functional differential equation
\begin{equation}\label{2.4}
\left\{\begin{array}{l}
\frac{d \tilde{u}(t)}{d t}=A(t, \omega, \tilde{u}_t), \quad t>0, \\
 \tilde{u}_{0}=\phi \in \mathcal{C}.
\end{array}\right.
\end{equation}
admits a global continuous mild solution $\tilde{u}(\cdot, \omega, \phi):[-\tau, \infty)\rightarrow \mathbb{R}^{n}$ which is $C^1$ in $t$ for any $\phi\in \mathcal{C}$ and P-a.e. $\omega\in \Omega$. Moreover, $\tilde{u}(t, \omega, \phi)$ is jointly continuous in $(t, \omega)$, $\tilde{u}(t, \cdot, \phi)$ is $\mathcal{F}_{t}$-measurable for any $\phi\in \mathcal{C}$ and the map $\tilde{u}_{t}(\cdot, \omega, \cdot): \mathcal{C }\rightarrow \mathcal{C}$ is compact for any $t>\tau$.
\end{lem}
\begin{proof}
Since $v(t, \omega, \cdot): \mathbb{R}^{n} \rightarrow \mathbb{R}^{n}$ and $v^{-1}(t, \omega, \cdot): \mathbb{R}^{n} \rightarrow \mathbb{R}^{n}$ are both continuous, then \eqref{2.4} can be written as
\begin{equation}\label{2.5}
\left\{\begin{array}{l}
\frac{d \tilde{u}(t)}{d t}=-C\tilde{u}(t)+v^{-1}(t, \omega, \cdot)\circ H f(v(t,\omega,\tilde{u}(t)))\\
+v^{-1}(t, \omega,\cdot)\circ Bg\left(v_t(\cdot,\omega,\tilde{u}_t)\right), \quad t>0, \\
 \tilde{u}_{0}=\phi \in \mathcal{C}.
\end{array}\right.
\end{equation}
For each $t \in \mathbb{R}^{+}$ and P-a.e. $\omega\in \Omega$. Moreover,  we have
\begin{equation}\label{2.5c}
\begin{aligned}
\|A(t, \omega, \tilde{u}_t)\|&\leq c|\tilde{u}(0)|+hL_f|\tilde{u}(0)|+bL_g\|\tilde{u}_t\|\\
& \leq(c+hL_f+bL_g)\|\tilde{u}_t\|,
\end{aligned}
\end{equation}
where $c,h,b$ represent  the norm of matrix $C,H$ and $B$ respectively. Hence,  the random retarded functional differential equation \eqref{2.4} admits a unique local solution with explosion time $\rho=\rho(\phi, \omega)$. We next show  that $\rho=\infty$.
By the variation-of-constants formula, we  see \eqref{2.4} is equivalent to the following integral equation
\begin{equation}\label{2.6a}
\left\{\begin{aligned}
\tilde{u}(t,\omega,\phi) &=S\ast\phi(t,\omega)+\int_{0}^{t} S(t-s,\omega)v^{-1}(s, \omega, \cdot)\circ
\\& [H f(v(s,\omega,\tilde{u}(s)))+ B\tilde{g}\left(v_s(\cdot,\omega,\tilde{u}_t)\right)] \mathrm{d} s, t \geq 0, \\
\tilde{u}(t) &=\phi(t), \quad t \in[-\tau, 0],
\end{aligned}\right.
\end{equation}
where $S\ast\phi(t,\omega)$ and $S(t,\omega)$ satisfy \eqref{2.8a} and \eqref{2.8b} respectively for $t \geq 0$ and P-a.e. $\omega\in \Omega$. Hence, we have
\begin{equation}\label{44}
\begin{aligned}
\left|\tilde{u}(t,\omega,\phi) \right|\leq &\left|\left(S* \phi\right)(t, \omega)\right|+\int_{0}^{t}|S(t-s)v^{-1}(s, \omega, \cdot)\circ \\
& [H f(v(s,\omega,\tilde{u}(s)))+ B\tilde{g}\left(v_s(\cdot,\omega,\tilde{u}_s)\right)] \mathrm{d} s| \mathrm{d} s \\
\leq &\left|\left(S* \phi\right)(t, \omega)\right|+\int_{0}^{t}S(t-s)[hL_f\|\tilde{u}(s, \omega,\phi)\|\\
&+bL_{\tilde{g}}\|\tilde{u}_s(\cdot, \omega,\phi)\|]\mathrm{d} s \\
\leq & K_0\mathrm{e}^{-\gamma t}\left\|\phi\right\|\\
&+ (hL_f+bL_{\tilde{g}})\int_{0}^{t} \mathrm{e}^{\varrho(t-s) / 2}\left\|\tilde{u}_s\left(\cdot, \omega, \phi\right)\right\| \mathrm{d} s, \\
\end{aligned}
\end{equation}
from which we can see for $t \geq \tau, \xi \in[-\tau, 0]$,
\begin{equation}\label{45}
\begin{aligned}
\|\tilde{u}(t+\xi, \omega,\phi) \|
\leq & K_0\mathrm{e}^{\gamma(\tau-t)}\left\|\phi\right\|+ (hL_f+bL_{\tilde{g}})\mathrm{e}^{-\varrho\tau / 2}\times\\
&\int_{0}^{t} \mathrm{e}^{\varrho / 2(t-s)}\left\|\tilde{u}_s\left(\cdot, \omega, \phi\right)\right\| \mathrm{d} s. \\
\end{aligned}
\end{equation}
Namely,
\begin{equation}\label{46}
\begin{aligned}
\mathrm{e}^{-\varrho t / 2}\left\|\tilde{u}_{t}\left(\cdot, \omega, \phi\right)\right\| \leq & c_{0} \mathrm{e}^{-\gamma t} \mathrm{e}^{-\varrho t / 2}\left\|\phi\right\| \\
&+c_{1} \int_{0}^{t} \mathrm{e}^{-\varrho s / 2}\left\|\tilde{u}_{s}\left(\cdot, \omega, \phi\right)\right\| \mathrm{d} s
\end{aligned}
\end{equation}
where $c_{0}=K_0\mathrm{e}^{\gamma \tau}$ and $ c_{1}= (hL_f+bL_{\tilde{g}})\mathrm{e}^{-\varrho\tau / 2}$.
Applying Gr\"{o}nwall's inequality yields
\begin{equation}\label{47}
\begin{aligned}
\mathrm{e}^{-\varrho t / 2}\left\|\tilde{u}_{t}\left(\cdot, \omega, \phi\right)\right\| \leq & c_{0}\|\phi\|[ \mathrm{e}^{-(\gamma+\varrho / 2)t}\\
&+c_{1}\mathrm{e}^{c_{1}t}  \int_{0}^{t} \mathrm{e}^{-(\gamma+\varrho / 2+c_1)s}\mathrm{d} s],
\end{aligned}
\end{equation}
and hence
\begin{equation}\label{48}
\begin{aligned}
\left\|\tilde{u}_{t}\left(\cdot, \omega, \phi\right)\right\| \leq & c_{0}\|\phi\|[ \mathrm{e}^{-\gamma t}+c_{1}\mathrm{e}^{(c_{1}+\varrho / 2)t}  \int_{0}^{t} \mathrm{e}^{-(\gamma+\varrho / 2+c_1)s}\mathrm{d} s]
\end{aligned}
\end{equation}
for P-a.e. $\omega \in \Omega$ and $0<t<\rho$. Thus, $\{\|\tilde{u}(t,\omega,\phi)\|: 0<t<\rho\}$ is bounded and so $\tilde{u}(\cdot,\omega,\phi): [0,\rho)\rightarrow H$ is uniformly continuous for P-a.e. $\omega \in \Omega$ and any $\phi\in \mathcal{C}$. Taking the limit
$\lim _{t \rightarrow \rho-}\tilde{u}_t(\cdot, \omega,\phi) \in \mathcal{C}$ as initial data for \eqref{2.4}, then by local existence we may extend $\tilde{u}$   to a larger interval than $[0, \rho)$. This contradicts the maximality of $\rho.$ Hence $\rho=\infty .$

It follows from the continuity of $A(t, \omega, \tilde{u}_t)$ with respect to $t$ that
$\tilde{u}(\cdot, \omega, \phi)$  is $C^1$ in $t$ for any $\phi\in \mathcal{C}$ and P-a.e. $\omega\in \Omega$. For each $t\in \mathbb{R}^{+}$ and $\phi\in \mathcal{C}$, $A(t, \omega, \tilde{u}_t)$ is $\mathcal{F}_{t}$-measurable. Hence,  a standard successive approximation argument shows that $\tilde{u}_t(\cdot, \omega,\phi)$ is jointly measurable in  $(t, \omega)$ and for each $t \in \mathbb{R}^{+}$, $\tilde{u}_t(\cdot, \omega,\phi)$ is $\mathcal{F}_{t}$-measurable for any $\phi\in \mathcal{C}$.

Furthermore if we fixed $ t_{0} \geq\rho$ then the above argument also shows that for P-a.e. $\omega \in \Omega$ the  set $\left\{d\tilde{u}(t, \omega, \phi) \mid: 0 \leq t \leq t_{0},\|\phi\| \leq 1\right\}$ is bounded.
Hence by Ascoli's theorem the family
\begin{equation}\label{3.9}
\left\{\tilde{u}_{t}(\cdot, \omega, \phi):\| \phi\| \leq 1\right\}
\end{equation}
is relatively compact in $\mathcal{C}$.
\end{proof}

Next, we show the existence of continuous measurable version of solutions to  the stochastic delayed HNNM \eqref{1.4}  constructed by the solution of \eqref{2.4}.
\begin{thm}\label{thm2}
The trajectory field $u(t,\omega,\phi)$, $t\in [-\tau,\infty)$  of the stochastic delayed HNNM \eqref{1.4} has a Borel-measurable version $U(\cdot,\cdot,\cdot): \mathbb{R}^{+}\times \Omega\times \mathcal{C}\mapsto \mathcal{C}$ with the following properties:\\
(I) For each $\phi  \in \mathcal{C}, U(t,\cdot,\phi)=u(t,\cdot,\phi)$  a.s. on $\Omega$.\\
(II) For each $t \in \mathbb{R}^{+}$ and $\phi \in \mathcal{C}, U(t,\cdot,\phi): \Omega\mapsto \mathcal{C}$ is $\mathcal{F}_{t}$-measurable.\\
(III) There is a Borel set $\Omega_{0} \subset \Omega$ of full P-measure such that for all $\omega \in \Omega_{0}$ the map
$U(\cdot, \omega, \cdot): \mathbb{R}^{+} \times \mathcal{C}\mapsto \mathcal{C}$ is continuous.\\
(IV) For each $t \in \mathbb{R}^{+}$ and every $\omega \in \Omega_{0}$, the map $U(t, \omega, \cdot): \mathcal{C}\mapsto \mathcal{C}$ is continuous.\\
(V) For each $t > \tau$ and all $\omega \in \Omega_{0}$, the map $U(t, \omega, \cdot): \mathcal{C} \rightarrow \mathcal{C}$ is compact.
\end{thm}
\begin{proof}
Define $u:[-\tau, \infty) \times \Omega \times \mathcal{C} \rightarrow \mathbb{R}^n$ by
\begin{equation}\label{3.11}
\begin{aligned}
u(t, \omega,  \phi)&=\left\{\begin{array}{ll}
v\left(t, \omega,\tilde{u}(t, \omega,\phi)\right) &\quad   t\geq 0, \\
 \phi(t) &\quad   t \in[-\tau, 0),
\end{array}\right.
\end{aligned}
\end{equation}
We prove in the following that the above defined $u$ is in fact the solution of the stochastic delayed HNNM \eqref{1.4}. It follows from Lemma \ref{lem5} that for each $t \in \mathbb{R}^{+}$ and $\phi\in \mathcal{C}$, $\tilde{u}(t, \omega, \phi)$ is  jointly measurable in $(t, \omega)$ and $\tilde{u}(t, \cdot, \phi)$ is $\mathcal{F}_t$-adapted, which combined with the $\mathcal{F}_t$-adaptivity of $v(t,\omega,\cdot)$ indicates   that the process $u(t,\cdot,\phi)$ is $\mathcal{F}_t$-adapted, for each $t \geq 0$. Since $\left\{\tilde{u}(t, \cdot, \phi): t \geq 0\right\}$ is $\mathcal{F}_{t}$-adapted and has $\mathrm{C}^{1}$ sample paths, by taking  Ito differentials of \eqref{3.11} and invoking Bismut's generalized Ito formula for stochastic flows, we have for any $t>0$
\begin{equation}\label{3.13}
\begin{aligned}
d u\left(t, \omega, \phi\right)=&v\left(t, \omega, \frac{d \tilde{u}(t, \omega, \phi)}{d t}\right) d t\\
&+\sigma\left(v\left(t, \cdot,\tilde{u}(t, \omega, \phi)\right)\right) d w(t)\\
=&-Cv(t,\omega, \tilde{u}(t, \omega, \phi))+H f(v(t,\omega, \tilde{u}(t, \omega, \phi))) \\ &+Bg(v_t(\cdot,\omega,\tilde{u})) +\sigma\left(v(t,\omega, \tilde{u}(t, \omega, \phi))\right) d w(t).
\end{aligned}
\end{equation}
Then, in the case $t>0$, we have
\begin{equation}\label{3.14}
\begin{aligned}
v_t(\cdot,\omega,\tilde{u}_t(\cdot, \omega, \phi))&=\left\{\begin{array}{ll}
v(t+\cdot, \omega, \tilde{u}_t(\cdot, \omega, \phi)),  t+\cdot \geq 0 \\ \phi(t+\cdot),  -\tau\leq t+\cdot<0
\end{array}\right.\\
&=u_t(\cdot, \omega,\phi),
\end{aligned}
\end{equation}
for any $s \in [-\tau, 0]$ and P-a.e. $\omega\in \Omega$.
Therefore \eqref{3.13} becomes
\begin{equation}\label{3.15}
\begin{aligned}
d u\left(t, \omega, \phi\right)
=&-Cu(t,\omega,\phi)+H f(u(t,\omega,\phi)) \\
&+Bg(u_t(\cdot,\omega,\phi)) +\sigma\left(u(t,\omega,\phi)\right) d w(t),
\end{aligned}
\end{equation}
i.e., $u(t,\omega,\phi)$ satisfies stochastic delayed HNNM \eqref{1.4} and
\begin{equation}\label{3.16}
\begin{aligned}
U(t, \omega,\phi) &:=v_t(\cdot,\omega,\Tilde{u}_t(\cdot, \omega, \phi))
\end{aligned}
\end{equation}
is the version of its trajectory field we are seeking.

From the construction of $U(t, \omega,\phi)$, one can easily see that statement (I) holds. It remains to show that $U(t, \omega,\phi)$ satisfies all the requirements of the theorem. We begin by proving that $U(t, \omega,\phi)$ is jointly (Borel)-measurable in $(t, \omega)$. By the definition of \eqref{2.3}, one can clear see that \begin{equation}\label{3.18}
\begin{aligned}
&\mathbb{R}^{+} \times \Omega \rightarrow L\left(\mathcal{C}\right) \\
&(t, \omega) \mapsto v_t(s, \omega, \cdot)
\end{aligned}
\end{equation}
is Borel-measurable   for each $s\in [-\tau,0]$. But the Borel field of $C\left([-\tau,0] \times \mathbb{R}^n, \mathbb{R}^n\right) $ is identified with that of $C([-\tau, 0],L(\mathbb{R}^n))$ and the latter is generated by evaluations, so it follows immediately that the map
\begin{equation}\label{3.19}
\begin{gathered}
\Lambda: \mathbb{R}^{+} \times \Omega \rightarrow  L\left(\mathcal{C}\right) \\
(t, \omega) \mapsto v_{t}\left(\cdot, \omega, \cdot\right)
\end{gathered}
\end{equation}
is Borel-measurable and has continuous sample paths. It follows from Lemma  \ref{lem5} that
$\tilde{u}_t(\cdot, \omega, \phi)$ is Borel-measurable on $\mathbb{R}^{+} \times \Omega \times \mathcal{C}$.  Therefore the composition
 \begin{equation}\label{3.21}
U(t, \omega,\phi) :=v_t(\cdot,\omega,\tilde{u}_t(\cdot, \omega, \phi))
 \end{equation}
 is measurable on $\mathbb{R}^{+} \times \Omega \times \mathcal{C}$ and $U\left(t, \omega, \cdot\right): \mathcal{C} \rightarrow \mathcal{C}$ is continuous  for $t \in \mathbb{R}^{+}$, P-a.e. $\omega\in \Omega$. Also for P-a.e. $\omega\in \Omega$, $U(\cdot, \omega, \cdot)$ is continuous on $\mathbb{R}^{+} \times \mathcal{C}$. This proves assertions (II), (III) and (IV) of the theorem.

At last, we prove statement (V) of the theorem. Fix $t > \tau$ and pick any universal Borel set $\Omega_{0} \subset \Omega$ of full $P$-measure such that $v(t, \omega, \cdot): \mathbb{R}^{n} \rightarrow \mathbb{R}^{n}$ is linear for all $\omega \in \Omega_{0}$. Take $\omega \in \Omega_{0}$,
then the composition \eqref{3.21} yields compactness of $ U(t, \omega, \cdot): \mathcal{C} \rightarrow \mathcal{C}$ as we have already established that the map $\tilde{u}_{t}(\cdot,\omega, \cdot): \mathcal{C }\rightarrow \mathcal{C }$ is compact for any $t>\tau$. This completes the proof of the theorem.
\end{proof}

In  Theorem \ref{thm2}, we construct a continuous measurable version $U(\cdot, \omega, \cdot)$ of \eqref{1.4} that generates a random semiflow. The idea originates from the early works \cite{M90}, \cite{MS2003} and \cite{MS2004}, where the authors  recasted the  stochastic delayed differential equations  into a Hilbert space to study the Lyapunov spectrum and invariant manifolds.  Our work can be regarded as a first attempt to study the random attractors of the HNNM \eqref{1.4} in Banach spaces.

\subsection*{Cocycle property}

The remaining part of this section is  devoted to showing the cocycle property of the trajectory field version $\mathrm{R}^{+} \times \Omega \times \mathcal{C} \rightarrow \mathcal{C}$ constructed in Theorem \ref{thm2} using   the  Wong-Zakai  approximation  of the Wiener process, i.e., a  piece-wise linear interpolation of its sample paths. See Twardowsky \cite{tw}.

The Ito type \eqref{2.1} is equivalent to the following Stratonovich equation
\begin{equation}\label{3.96}
\left\{\begin{array}{l}
d \hat{v}(t,\cdot,\cdot)=\sum_{i=1}^{m}\Sigma\circ \hat{v}(t,\cdot,\cdot) \circ d w_i(t)\\
-\frac{1}{2} \sum_{i=1}^{m}\left(\Sigma \circ\left([\Sigma \circ \hat{v}(t, \cdot, \cdot)](\cdot)\left(e_{i}\right)\right\}\right)(\cdot)\left(e_{i}\right) d t, \\
\hat{v}(0, \cdot, \cdot)=\mathrm{id}_{\mathbb{R}^{n}},
\end{array}\right.
\end{equation}
where the notation $\left[\Sigma \circ  A\right](\cdot)\left(e_{i}\right), A \in L\left(\mathbb{R}^{n}\right)$  stands for the map
$$
\mathbb{R}^{n} \rightarrow \mathbb{R}^{n}: v\rightarrow \Sigma \circ A(v)\left(e_{i}\right).
$$
For each $k \geq 1$,  let $\hat{v}^{k}: \mathbb{R}^{+} \times \Omega \rightarrow L\left(\mathbb{R}^{n}\right)$ be the fundamental  
solution of the following random differential equation+
\begin{equation}\label{3.97}
\left\{\begin{array}{l}
d \hat{v}^{k}(t,\cdot,\cdot)=\sum_{i=1}^{m}\left[\Sigma {\circ} \hat{v}^{k}(t, \cdot, \cdot)\right](\cdot)\left(e_{i}\right)\left(W_{i}^{k}\right)^{\prime}(t) d t\\
-\frac{1}{2} \sum_{i=1}^{m}\left(\Sigma \circ\left([\Sigma \circ \hat{v}^{k}(t, \cdot, \cdot)](\cdot)\left(e_{i}\right)\right\}\right)(\cdot)\left(e_{i}\right) d t, \\
\varphi(0, \cdot, \cdot)=\mathrm{id}_{\mathbb{R}^{n}}.
\end{array}\right.
\end{equation}
Here $W_{i}^{k}(t)$ is the piece-wise linear approximation  of $\omega_i, i=1,2,\cdots,m$, which is defined by
$$
\left(W_i^k\right)^{\prime}(t, \omega)=\left(\omega_i^k\right)^{\prime}(t), \quad \omega \in \Omega, \quad t \geq 0,
$$
where
$$
\omega_i^k(t)=k\left[\omega\left(\frac{j+1}{k}\right)-\omega_i\left(\frac{j}{k}\right)\right]\left(t-\frac{j}{k}\right)+\omega_i\left(\frac{j}{k}\right),
$$
for $\quad \frac{j}{k} \leq t<\frac{j+1}{k}, \quad j=0,1,2, \ldots$ and $\omega_i$ is standard mutually independent Brownian motions. It follows from \cite{IW} (see also Mohamed ref  \cite{M90} and Twardowsky \cite{tw}) that for a.a. $\omega \in \Omega$,
\begin{equation}\label{3.98}
\lim _{k \rightarrow \infty} v^{k}(t, \omega, \cdot)=v(t, \omega, \cdot)
\end{equation}
in $L\left(\mathbb{R}^{n}\right)$ uniformly in $t \in[0, T]$, for $0<T<\infty$.

\begin{thm}\label{thm3}
For each integer $k \geq 1$, let
$\hat{v}^k(t, \omega, \cdot): \mathbb{R}^{n} \rightarrow \mathbb{R}^{n}$ be the solution of \eqref{3.97} and define the following operator
\begin{equation}\label{25}
\begin{aligned}
A^k(t,\omega, \phi)\triangleq &[\hat{v}^k(t, \omega, \cdot)]^{-1} [-C\hat{v}^k(t,\omega,\phi(0))+\\
&H f(\hat{v}^k(t,\omega,\phi(0))) +Bg\left(\hat{v}^k_t(\cdot,\omega,\phi)\right)],
\end{aligned}
\end{equation}
with $\hat{v}^k_t(\cdot, \omega, \cdot): [-\tau,0]\times \mathcal{C}\rightarrow \mathcal{C}$ being defined by
\begin{equation}\label{26}
\begin{aligned}
\hat{v}^k_t(s, \omega, \phi)&=\left\{\begin{array}{ll}
\hat{v}^k(t+s, \omega, \phi(0)): & t+s\geq 0, \\
\phi(t+s) & -\tau \leq t+s<0.
\end{array}\right.\\
\end{aligned}
\end{equation}
Let $0<T<\infty$. Denote by $B$  a bounded set in $\mathcal{C}$ and $\Tilde{u}^k(t, \omega, \cdot):[-\tau, \infty)\rightarrow \mathbb{R}^{n}$ the solution of the following random retarded  partial  differential equation
\begin{equation}\label{27}
\left\{\begin{array}{l}
\frac{d \Tilde{u}^k(t)}{d t}=A^k(t, \omega, \Tilde{u}^k_t), \quad t>0, \\
 \Tilde{u}^k_{0}=\phi \in \mathcal{C}.
\end{array}\right.
\end{equation}
Define
\begin{equation}\label{3.50}
\begin{aligned}
U^k(t, \omega,\phi) &:=v^k_t(\cdot,\omega,\tilde{u}^k_t(\cdot, \omega, \phi)),
\end{aligned}
\end{equation}
then $U^k(t, \omega,\phi)$ satisfies
$$
\lim _{k \rightarrow \infty} \sup _{0 \leq t \leq T} \sup _{\phi\in B}\left\|U^k(t, \omega, \phi)-U(t, \omega, \phi )\right\|=0
$$
for P-a. e. $\omega \in \Omega$.
\end{thm}
\begin{proof}
Fix $0<T<\infty$ and $B\subset\mathcal{C}$ a bounded set throughout the proof. By Lemma \ref{lem5}, \eqref{27} admits a global solution for P-a. e. $\omega \in \Omega$, which is defined as follows.
\begin{equation}\label{28}
\left\{\begin{aligned}
\Tilde{u}^k(t) &=S(t)\ast\phi+\int_{0}^{t} S(t-s)[\hat{v}^k(s, \omega, \cdot)]^{-1}\circ\\
& [H f(\hat{v}^k(s,\omega,\Tilde{u}^k(s)))+ B\tilde{g}\left(\hat{v}^k_s(\cdot,\omega,\Tilde{u}^k_t)\right)] \mathrm{d} s, t \geq 0, \\
u(t) &=\phi(t), \quad t \in[-\tau, 0].
\end{aligned}\right.
\end{equation}
Therefore, by denoting $I(t,\omega,\phi)=\|U^k(t, \omega, \phi)-U(t, \omega, \phi )\|$, we have for $P$-a.e. $\omega \in \Omega$,
\begin{equation}\label{29}
\begin{aligned}
I(t,\omega,\phi)
=&\|v_t(\cdot,\omega,\tilde{u}_t(\cdot, \omega, \phi))-v^k_t(\cdot,\omega,\tilde{u}^k_t(\cdot, \omega, \phi))\|
\\ \leq &\|v_t(\cdot,\omega,\tilde{u}_t(\cdot, \omega, \phi))-v_t(\cdot,\omega,\tilde{u}^k_t(\cdot, \omega, \phi))\|
\\+&\|v_t(\cdot,\omega,\tilde{u}^k_t(\cdot, \omega, \phi))-v^k_t(\cdot,\omega,\tilde{u}^k_t(\cdot, \omega, \phi))\|
\\ \leq &\|v_t(\cdot,\omega,\tilde{u}_t(\cdot, \omega, \phi)-\tilde{u}^k_t(\cdot, \omega, \phi))]\|
\\+&\|[v_t(\cdot,\omega,\cdot)-v^k_t(\cdot,\omega,\cdot)]\tilde{u}^k_t(\cdot, \omega, \phi)]\|
\\
\leq & \|v_t(\cdot,\omega,\cdot)\|\|\tilde{u}_t(\cdot, \omega, \phi)-\tilde{u}^k_t(\cdot, \omega, \phi)\|
\\+&\|v_t(\cdot,\omega,\cdot)-v^k_t(\cdot,\omega,\cdot)\|\|\tilde{u}^k_t(\cdot, \omega, \phi)\|
\\
\triangleq & \|v(t,\omega,\cdot)\|I_1+I_2\|\tilde{u}^k_t(\cdot, \omega, \phi)\|.
\end{aligned}
\end{equation}
Since $\lim _{k \rightarrow \infty} v^{k}(t, \omega, \cdot)=v(t, \omega, \cdot)$, we have $I_2\rightarrow 0$ as $k\rightarrow \infty$. Because $\|v(t,\omega,\cdot)\|$ is bounded, we only need to  estimate  $I_1$.
\begin{equation}\label{29a}
\begin{aligned}
 |I_1|=&|\int_{0}^{t} S(t-s)[v^{-1}(s, \omega, \cdot)\circ H f(v(s,\omega,u(s)))-\\
&\hat{v}^k(s, \omega, \cdot)^{-1}\circ H f(\hat{v}^k(s,\omega,\Tilde{u}^k(s)))]|
\\&+| \int_{0}^{t} S(t-s)[v^{-1}(s, \omega, \cdot)\circ B\tilde{g}\left(v_s(\cdot,\omega,u_t)\right)-\\&
\hat{v}^k(s, \omega, \cdot)^{-1}\circ B\tilde{g}\left(\hat{v}^k_s(\cdot,\omega,\Tilde{u}^k_t)\right)] \mathrm{d} s|
\\&
\triangleq I_{11}+I_{12}.
\end{aligned}
\end{equation}
 We estimate each term on the right hand side of \eqref{29a} in the following by Lemmas \ref{lem5}. By the $\mathbf{Hypothesis \  A1}$, we can see
\begin{equation}\label{30}
\begin{aligned}
I_{11}\leq&\|\int_{0}^{t} S(t-s)v^{-1}(s, \omega, \cdot)\circ H[ f(v(s,\omega,u(s)))\\
&-f(\hat{v}^k(s,\omega,\Tilde{u}^k(s)))]\mathrm{d} s\|+\|\int_{0}^{t} S(t-s)[v^{-1}(s, \omega, \cdot)\\
&-\hat{v}^k(s, \omega, \cdot)^{-1}]\circ H f(\hat{v}^k(s,\omega,\Tilde{u}^k(s)))\mathrm{d} s\|\\&
\triangleq I_{111}+I_{112}.
\end{aligned}
\end{equation}
Since $f$ is bounded, we have for any $\rho \in \mathbb{R}^n$, $Hf(\rho)\leq M$ which together with  \eqref{3.98}, implies that for any $\varepsilon>0$, there exists sufficient large $K_0$ such that for any $k>K_0$
\begin{equation}\label{31}
\begin{aligned}
I_{112}\leq M\varepsilon.
\end{aligned}
\end{equation}
Moreover,
\begin{equation}\label{32}
\begin{aligned}
I_{111}\leq&\|\int_{0}^{t} S(t-s)[v^{-1}(s, \omega, \cdot)\circ H[ f(v(s,\omega,u(s)))\\
&-f(\hat{v}^k(s,\omega,u(s)))]\mathrm{d} s\|+\|\int_{0}^{t} S(t-s)[v^{-1}(s, \omega, \cdot)\\
& \circ H[ f(\hat{v}^k(s,\omega,u(s)))-f(\hat{v}^k(s,\omega,\Tilde{u}^k(s)))]\mathrm{d} s\|\\&
\leq CKL_f \int_{0}^{t} \mathrm{e}^{\varrho(t-s) / 2}\|v\left(s, \omega, u(s)\right)\\&
-\hat{v}^k(s,\omega,u(s))\|\mathrm{d} s
+KL_f \int_{0}^{t} \mathrm{e}^{\varrho(t-s) / 2}\\& \left\|u(s)-\Tilde{u}^k(s)\right\|\mathrm{d} s \triangleq I_{1111}+I_{1112}.
\end{aligned}
\end{equation}
It follows from  \eqref{3.98} that for any $\varepsilon>0$, there exists sufficient large $K_0$ such that for any $k>K_0$
\begin{equation}\label{33}
\begin{aligned}
I_{1111}\leq CKL_f \varepsilon \int_{0}^{t} \mathrm{e}^{\varrho(t-s) / 2}\mathrm{d} s\leq CKL_f \varepsilon.
\end{aligned}
\end{equation}
Thus, we have for $k>K_0$
\begin{equation}\label{34}
\begin{aligned}
I_1\leq & M\varepsilon+CKL_f \varepsilon+KL_f \int_{0}^{t} \mathrm{e}^{\varrho(t-s) / 2}\left\|u(s)-\Tilde{u}^k(s)\right\|\mathrm{d} s,
\end{aligned}
\end{equation}
which together with the Gr\"{o}nwall inequality indicates that $I_1\rightarrow 0$ as $k\rightarrow \infty$.
\end{proof}

We are now in the position to provide the cocycle property of $U(\cdot,\cdot,\cdot): \mathbb{R}^{+}\times \Omega\times \mathcal{C}\mapsto \mathcal{C}$. We directly give the following results for which the proof is the same as Theorem 3 in  \cite{M90}.

\begin{thm}\label{thm4} (The Cocycle Property)
Let $\theta: \mathbb{R}^{+} \times \Omega \rightarrow \Omega$ be the shift
corresponding Wiener measure on $(\Omega, \mathcal{F})$ which is defined by
$$\theta_t \omega(\cdot)=\omega(\cdot+t)-\omega(t), t \in \mathbb{R}.$$
 Then, there is a universal set $\Omega_0\subset \Omega$ of full Wiener measure such that
$$
U\left(t_{2}, \theta_{t_{1}}\omega, U\left(t_{1}, \omega,\phi\right)\right)=U\left(t_{1}+t_{2}, \omega,\phi\right)
$$
for all $\omega \in \Omega_0, \phi\in \mathcal{C}$ and $t_{1}, t_{2} \geq 0$.
\end{thm}

\section{Random  attractors}\label{S2.2:ue}

In this section, we are concerned about the existence of random attractors of \eqref{1.4}, which relies on showing the existence of an absorbing set and the asymptotic compactness of the RDS $(\theta, U)$. Throughout the reminder of this paper,  we denote the family of  tempered sets in the state space $C$ by $\mathcal{D}$. We first give the following uniform estimate results.
\begin{lem}\label{lem6}
Assume that  conditions of Lemma \ref{lem5}, Theorem \ref{thm3} hold and
$$
c_1+\varrho/2<0<c_1+\varrho/2+\gamma,
$$
then $U$ possesses a random absorbing set $\{K(\omega)\}_{\omega \in \Omega} \in \mathcal{D}$, that is, for any $\{B(\omega)\}_{\omega \in \Omega} \in \mathcal{D}$ and $\mathbb{P}$-a.e. $\omega \in \Omega$, there exists $t_{B}(\omega)>0$ such that
$$
U\left(t, \theta_{-t} \omega, \phi\left(\theta_{-t} \omega\right)\right) \in K(\omega) \quad \text { for all } t \geq t_{B}(\omega)
$$
for all $t \geq t_{B}(\omega)$ and $\phi\left(\theta_{-t} \omega\right) \in B\left(\theta_{-t} \omega\right)$.
\end{lem}
\begin{proof}
For any  $\phi(\omega) \in B(\omega)$, it follows from \eqref{47} in the proof of Lemma \ref{lem5} that
\begin{equation}\label{47a}
\begin{aligned}
\mathrm{e}^{-\varrho t / 2}\left\|\tilde{u}_{t}\left(\cdot, \omega, \phi\right)\right\| \leq & c_{0}\|\phi\|[ \mathrm{e}^{-(\gamma+\varrho / 2)t}\\
&+c_{1}\mathrm{e}^{c_{1}t}  \int_{0}^{t} \mathrm{e}^{-(\gamma+\varrho / 2+c_1)s}\mathrm{d} s],
\end{aligned}
\end{equation}
where $c_{0}=l\mathrm{e}^{\gamma \tau}$ and $ c_{1}= (hL_f+bL_{\tilde{g}})\mathrm{e}^{-\varrho \tau / 2}$. This implies
\begin{equation}\label{48a}
\begin{aligned}
\left\|\tilde{u}_{t}\left(\cdot, \omega, \phi\right)\right\| \leq & c_{0}\|\phi\|[ \mathrm{e}^{-\gamma t}+c_{1}\mathrm{e}^{(c_{1}+\varrho / 2)t}  \int_{0}^{t} \mathrm{e}^{-(\gamma+\varrho / 2+c_1)s}\mathrm{d} s]
\end{aligned}
\end{equation}
for P-a.e. $\omega \in \Omega$ and $0<t<\rho$ by  Gr\"{o}nwall's inequality. Thus, we have
\begin{equation}\label{48b}
\begin{aligned}
\left\|U\left(t, \omega, \phi\right)\right\|=&\|v_t(\cdot,\omega,\tilde{u}_t(\cdot, \omega, \phi))\|\\
\leq & c_{0}\|\phi\|\|v_t(\cdot,\omega,\cdot)\|\mathrm{e}^{-\gamma t}+c_{0}c_{1}\mathrm{e}^{(c_{1}+\varrho / 2)t}\|\phi\|\times \\
  &\|v_t(\cdot,\omega,\cdot)\|\int_{0}^{t} \mathrm{e}^{-(\gamma+\varrho / 2+c_1)s}\mathrm{d} s
\end{aligned}
\end{equation}
for P-a.e. $\omega \in \Omega$. Therefore,
\begin{equation}\label{49a}
\begin{aligned}
\left\|U\left(t, \theta_{-t} \omega, \phi\right)\right\|=&\|v_t(\cdot,\theta_{-t} \omega,\tilde{u}_t(\cdot, \theta_{-t} \omega, \phi))\|\\
\leq & c_{0}\mathrm{e}^{-\gamma t}\|\phi(\theta_{-t} \omega)\|\|v_t(\cdot,\theta_{-t} \omega,\cdot)\|+c_{0}c_{1}\times \\
&(\frac{\mathrm{e}^{(c_{1}+\varrho / 2)t}}{c_1+\varrho/2+\gamma}-\frac{\mathrm{e}^{-\gamma t}}{c_1+\varrho/2+\gamma})\times \\
  &\|\phi(\theta_{-t} \omega)\| \|v_t(\cdot,\theta_{-t} \omega,\cdot)\|.
\end{aligned}
\end{equation}
Since $c_{1}+\varrho / 2<0, \gamma>0$, there must exits a $\lambda$ and $T_B>0$ such that for all $t>T_B$
$$
c_{0}\mathrm{e}^{-\gamma t}+c_{0}c_{1}(\frac{\mathrm{e}^{(c_{1}+\varrho / 2)t}}{c_1+\varrho/2+\gamma}-\frac{\mathrm{e}^{-\gamma t}}{c_1+\varrho/2+\gamma})<\lambda.
$$
Moreover, since $\|v_t(\cdot, \omega,\cdot)\|$ is bounded for all $t\in \mathbb{R}^+$ and P-a.e. $\omega \in \Omega$, if $\phi\left(\theta_{-t} \omega\right) \in B\left(\theta_{-t} \omega\right)$
 then  for all $t \geq t_{B}(\omega)$,
$$
\left\|U\left(t, \theta_{-t} \omega\right), \phi\left(\theta_{-t} \omega\right)\right\| \leq \lambda \|\phi(\theta_{-t} \omega)\| \|v_t(\cdot,\theta_{-t} \omega,\cdot)\|.
$$
Therefore, $K(\omega)=\left\{\vartheta \in \mathcal{C}:\|\vartheta\| \leq \lambda \|\phi(\theta_{-t} \omega)\| \|v_t(\cdot,\theta_{-t} \omega,\cdot)\|\right\}$ form a tempered bounded closed random absorbing set of RDS $U$
in $\mathcal{D}$, which completes the proof.
\end{proof}

In the next lemma, we show that we can take initial data in $\{K(\omega)\}_{\omega \in \Omega}$ to obtain the pullback asymptotic compactness of $U$.
\begin{lem}\label{lem7}
Assume that $\{K(\omega)\}_{\omega \in \Omega}$ is a random absorbing in $\mathcal{C}$ for the RDS $U$. Then $(\theta, U)$ is pullback asymptotically compact.
\end{lem}
\begin{proof}
Take an arbitrary random set $\{B(\omega)\}_{\omega \in \Omega} \in \mathcal{D}$, a sequence $t_{n} \rightarrow+\infty$ and $\phi_n \in B\left(\theta_{-t_{n}} \omega\right)$. We have to prove that $\left\{U\left(t_{n}, \theta_{-t_{n}} \omega, \phi_n\right)\right\}$ is precompact. Since $\{K(\omega)\}$ is a random absorbing for $\Psi$, then there exists $T>0$ such that, for all $\omega \in \Omega$,
\begin{equation}\label{4.24}
U\left(t, \theta_{-t} \omega\right) B\left(\theta_{-t} \omega\right) \subset K(\omega)
\end{equation}
for all $t \geq T$.
Because $t_{n} \rightarrow+\infty$, we can choose $n_{1} \geq 1$ such that $t_{n_{1}}-1 \geq T$. Applying \eqref{4.24} for $t=t_{n_{1}}-1$ and $\omega=\theta_{-1} \omega$, we find that
\begin{equation}\label{4.25}
\eta_{1} \triangleq \Psi\left(t_{n_{1}}-1, \theta_{-t_{n_{1}}} \omega, \phi_{n_{1}}\right)  \in K\left(\theta_{-1} \omega\right)
\end{equation}
Similarly, we can choose a subsequence $\left\{n_{k}\right\}$ of $\{n\}$ such that $n_{1}<n_{2}<\cdots<n_{k} \rightarrow$ $+\infty$ such that
\begin{equation}\label{4.26}
\eta_{k} \triangleq \Psi\left(t_{n_{k}}-k, \theta_{-t_{n_{k}}} \omega, \phi_{n_{k}}\right)  \in K\left(\theta_{-k} \omega\right)
\end{equation}
Hence, by the assumption we conclude that
the sequence
\begin{equation}\label{4.27}
\left\{U\left(k, \theta_{-k} \omega, \eta_{k}\right)\right\}
\end{equation}is precompact.
On the other hand by \eqref{4.26}, we have
\begin{equation}\label{4.28}
\begin{aligned}
U(k, \theta_{-k} \omega, \eta_{k}) &=U(k, \theta_{-k} \omega,U(t_{n_{k}}-k, \theta_{-t_{n_{k}}} \omega, U_{n_{k}}) \\
&=U\left(t_{n_{k}}, \theta_{-t_{n_{k}}} \omega,\phi_{n_{k}}\right),
\end{aligned}
\end{equation}
for all $k \geq 1$. Combining \eqref{4.27} and \eqref{4.28}, we obtain that the sequence $\left\{U\left(t_{n_{k}}, \theta_{-t_{n_{k}}} \omega,\phi_{n_{k}}\right)\right\}$ is precompact. Therefore,  $\left\{U\left(t_{n}, \theta_{t_{n}} \omega,\phi_{n_{k}}\right) \right\}$ is precompact, which completes the proof.
\end{proof}

Lemma \eqref{lem6} says that the continuous RDS $\Psi$ has a random absorbing set while Lemma \eqref{lem7} tells us that $(\theta, U)$ is pullback asymptotically compact in $\mathcal{C}$. Thus, it follows from Lemma \ref{lem1} that the continuous RDS $(\theta, U)$ possesses a random attractor. Namely, we obtain the following result.

\begin{thm}\label{thm5.1} Assume that conditions of Lemma \ref{lem6} hold,  then the continuous RDS $U$ generated by  \eqref{1.4} admits a unique $\mathcal{D}$-pullback attractor in $\mathcal{C}$ belonging to the class $\mathcal{D}$.
\end{thm}

\section{Exponentially attracting stationary solution}

In this section, we   derive sufficient conditions that guarantee  the random attractor is  an exponentially attracting random fixed point $\xi^{*}$ by a  general Banach fixed point theorem. n particular we use the following fixed point theorem in probability analysis, which is often  called the uniform strictly contracting property,  originates from the early work \cite{B76} and  was then extended to various versions in \cite{CKS04,PL12,KY21} and \cite{DLS03}. Here, we adopt the version in \cite{DLS03} for  infinite dimensional case.
\begin{lem}\label{lem6.1}  Let $(\mathcal{G}, d_{\mathcal{G}})$ be a complete metric space with bounded metric. Suppose that
$$
U(t, \omega, \mathcal{G}) \subset \mathcal{G}
$$
for $\omega \in \Omega, t \geq 0$, and that $x \rightarrow U(t, \omega, x)$ is continuous. In addition, we assume the contraction condition: There exists a constant $k<0$ such that, for $\omega \in \Omega$,
$$
\sup _{x \neq y \in \mathcal{G}} \log \frac{d_{\mathcal{G}}(U(1, \omega, x), U(1, \omega, y))}{d_{\mathcal{G}}(x, y)} \leq k.
$$
Then $U$ has a unique generalized fixed point $\gamma^{*}$ in $\mathcal{G}$. Moreover, the following convergence property holds:
$$
\lim _{t \rightarrow \infty} U\left(t, \theta_{-t} \omega, x\right)=\gamma^{*}(\omega)
$$
for any $\omega \in \Omega$ and $x \in \mathcal{G}$.
\end{lem}

\begin{thm}\label{Theorem 6} Assume that conditions of Lemma \ref{lem6} hold. Moreover, assume that
$$
\frac{\varrho}{2}e^{[\frac{\varrho}{2}+L_v  (hL_f+bL_{\tilde{g}})]}+\tau-1<0.
$$Then the RDS $U$ generated by SNDRDE \eqref{1.4} possess a tempered random fixed point $\chi^{*}$, which is unique under all tempered random variables in $\mathcal{C}$  and  attracts exponentially fast every random variable in $\mathcal{C}$.
\end{thm}
\begin{proof}
If $c_1+\varrho/2<0<c_1+\varrho/2+\gamma$, then the conditions of Theorem \ref{thm5.1} hold and hence \eqref{1.4} possess random attractors in $\mathcal{C}$. We will prove that  \eqref{1.4} admits a unique globally exponentially attracting random stationary solution in $\mathcal{C}$, which immediately implies the random attractor in $\mathcal{C}$ obtained in Theorem \ref{thm5.1} is the random fixed point $\chi^{*}$.

By  \eqref{48b}, we can see that for any $\phi \in \mathcal{C}$
\begin{equation}\label{48c}
\begin{aligned}
\left\|U\left(t, \omega, \phi\right)\right\|
\leq & c_{0}\|\phi\|\|v_t(\cdot,\omega,\cdot)\|\mathrm{e}^{-\gamma t}+c_{0}c_{1}\mathrm{e}^{(c_{1}+\varrho / 2)t}\|\phi\|\times \\
  &\|v_t(\cdot,\omega,\cdot)\|\int_{0}^{t} \mathrm{e}^{-(\gamma+\varrho / 2+c_1)s}\mathrm{d} s
\end{aligned}
\end{equation}
which implies that for any $\phi \in \mathcal{C}_{co}$, $U(t,\omega,\psi )\in \mathcal{C}$, i.e. $\mathcal{C}$ is invariant under the random semiflow $U$. Moreover, it follows from Lemma \ref{lem5} that $U$ is continuous in $\mathcal{C}$.

Therefore, we only need to prove the contraction property. That is, there exists $k<0$ such that
 \begin{equation}\label{6.1}
\begin{aligned}
\sup _{\varphi \neq \psi \in \mathcal{C}}\|U(1, \omega, \varphi)-U(1, \omega, \psi)\|\leq e^k\|\varphi-\psi\|.
\end{aligned}
\end{equation}
Hence, it suffices to prove that for any $\varphi,\psi\in \mathcal{C}$
 \begin{equation}\label{6.2}
\begin{aligned}
\|U(1, \omega, \varphi)-U(1, \omega, \psi)\|=&\|v_t(\cdot,\omega,\tilde{u}_t(\cdot, \omega, \varphi))-\\
&v_t(\cdot,\omega,\tilde{u}_t(\cdot, \omega, \psi))\|
\leq  e^k\|\varphi-\psi\|.
\end{aligned}
\end{equation}

By Eq. \eqref{2.6a} and Eq. \eqref{3.21}, we have for any $\varphi,\psi\in \mathcal{C}$
 \begin{equation}\label{6.3}
\begin{aligned}
\|G(t,\omega)\|&\leq \|v_t(\cdot,\omega,\tilde{u}_t(\cdot, \omega, \varphi)-\tilde{u}_t(\cdot, \omega, \psi))\|\\
&\leq
\|v_t(\cdot,\omega,\cdot)\|\{\|S(t,\omega)\ast[\varphi- \psi]\|+\\
&\sup _{\zeta\in[-\tau,0]} \int_{0}^{t+\zeta} S(t+\zeta-r,\omega)v^{-1}(s, \omega, \cdot)\circ\\
& [H(f(v_s(\cdot,\omega,\Tilde{u}_s(\cdot,\omega,\varphi)))-f(v_s(\cdot,\omega,\Tilde{u}_s(\cdot,\omega,\psi))))\\
&+ B(\tilde{g}\left(v_s(\cdot,\omega,\Tilde{u}_s(\cdot,\omega,\varphi))\right)-\\
&\tilde{g}\left(v_s(\cdot,\omega,\Tilde{u}_s(\cdot,\omega,\psi))\right))] \mathrm{d} s\}\\
& \leq L_v e^{\gamma \tau}e^{-\gamma t}\|\phi-\psi\|+L_v e^{\frac{\varrho \tau}{2}} (hL_f+bL_{\tilde{g}})\times\\
& \int_{0}^{t} e^{-\frac{\varrho }{2} (t-r)}\left\|U(r, \omega,\phi)-U(r, \omega,\psi)\right\|\mathrm{d}r.
\end{aligned}
\end{equation}
Multiply both sides of \eqref{6.3} by $e^{\frac{\varrho t}{2}}$ and denote by $H(t,\omega)=e^{\frac{\varrho t }{2}}G(t,\omega)$, $c_3(t)=L_v e^{\gamma \tau}e^{-\gamma t}e^{\frac{\varrho t}{2}}$, $c_4(t)=L_v e^{\frac{\varrho (\tau+t)}{2}} (hL_f+bL_{\tilde{g}})$, we obtain
 \begin{equation}\label{6.3b}
\begin{aligned}
\|H(t,\omega)\|& \leq  c_3(t)\|\phi-\psi\|+c_4(t)\int_{0}^{t}\|H(r,\omega)\|\mathrm{d}r.
\end{aligned}
\end{equation}
Again, the Gr\"{o}nwall inequality gives rise to
 \begin{equation}\label{6.4}
\begin{aligned}
\|H(t,\omega)\|& \leq \|\phi-\psi\|(c_3(t)+\int_{0}^{t}c_3(r) e^{\int_{r}^{t}c_4(s)ds}\mathrm{d}r).
\end{aligned}
\end{equation}
Dividing both sides of \eqref{6.4} by $e^{\frac{\varrho t}{2}}$ and take $t=1$ lead to
 \begin{equation}\label{6.5}
\begin{aligned}
\|G(1,\omega)\|& \leq \|\phi-\psi\|(c_3(1)e^{-\frac{\varrho}{2}}+e^{-\frac{\varrho}{2}}\int_{0}^{1}c_3(r) e^{\int_{r}^{1}c_4(s)ds}\mathrm{d}r).
\end{aligned}
\end{equation}
Therefore, if there exists $\varsigma<0$ such that $c_3(1)e^{-\frac{\varrho}{2}}<e^{\varsigma}$ and $e^{-\frac{\varrho}{2}}\int_{0}^{1}c_3(r) e^{\int_{r}^{1}c_4(s)ds}\mathrm{d}r<e^{\varsigma}$ then result of the theorem hold. By a  simple computation, we  see that $\frac{\varrho}{2}e^{[\frac{\varrho}{2}+L_v  (hL_f+bL_{\tilde{g}})]}+\tau-1$ $<$ $0$ implies the above claim holds and hence completes the proof.
\end{proof}

\section{Numerical Simulations}
This section is devoted to some numerical simulations in order to demonstrate the effectiveness and efficiency of the established theoretical results. In our numerical examples, we consider the following stochastic delayed HNNM.
 \begin{equation}\label{7.1}
\left\{\begin{aligned}
d u_1(t)=& [-c_1 u_1(t)+h_{11} f_1\left(u_1(t)\right)+h_{12} f_2\left(u_2(t)\right)+ \\
& b_{11} g_1\left(u_1\left(t-\tau_1\right)\right)+b_{12} g_2\left(u_2\left(t-\tau_2\right)\right)]dt\\
&+\sigma_{11} u_1(t) d w_1(t) +\sigma_{12} u_2(t)  d w_2(t) \\
d u_2(t)=& [-c_2 u_2(t)+h_{21} f_1\left(u_1(t)\right)+h_{22} f_2\left(u_2(t)\right)+ \\
& b_{21} g_1\left(u_1\left(t-\tau_1\right)\right)+b_{22} g_2\left(u_2\left(t-\tau_2\right)\right)]dt\\
&+\sigma_{21} u_1(t)  d w_1(t) +\sigma_{22} u_2(t)  d w_2(t).
\end{aligned}\right.
\end{equation}
Take
$$u(t)=\left(u_{1}(t), u_{2}(t)\right)^{\mathrm{T}},$$
$$ u_t(\tau)=\left(u_{1}\left(t-\tau_{1}\right), u_{2}\left(t-\tau_{2}\right)\right)^{\mathrm{T}},$$
$$C=\left(\begin{array}{ll}
c_{1} & 0 \\
0 & c_{2}
\end{array}\right)=\left(\begin{array}{ll}
5 & 0 \\
0 & 5
\end{array}\right),$$

$$
\begin{aligned}
&H=\left(\begin{array}{ll}
h_{11} & h_{12} \\
h_{21} & h_{22}
\end{array}\right)=\left(\begin{array}{ll}
0.2 & 0.1 \\
0.3 & 0.1
\end{array}\right),
\end{aligned}
$$

$$
\begin{aligned}
&B=\left(\begin{array}{ll}
b_{11} & b_{12} \\
b_{21} & b_{22}
\end{array}\right)=\left(\begin{array}{cc}
-0.3 & 0.2 \\
0.1 & 0.3
\end{array}\right),
\end{aligned}
$$
$$
\begin{aligned}
&\sigma=\left(\begin{array}{ll}
\sigma_{11} & \sigma_{12} \\
\sigma_{21} & \sigma_{22}
\end{array}\right)=\left(\begin{array}{cc}
0.01 & 0 \\
0 & 0.02
\end{array}\right),
\end{aligned}
$$
$f_j(u_j)=g_j(u_j)=\tanh (u_j)=\frac{e^{u_j}-e^{-u_j}}{e^u_j+e^{-u_j}}$
for $j=1,2$.

Hence the characteristic equation of \eqref{7.1} is
\begin{equation}\label{7.2}
\begin{array}{l}
\operatorname{det}\left[\left(\begin{array}{ll}
\lambda-5-0.3e^{-\lambda\tau_1}  & 0.2e^{-\lambda\tau_1}\\
0.1e^{-\lambda\tau_2} & \lambda-5+0.3e^{-\lambda\tau_2}
\end{array}\right)\right]=0,
\end{array}
\end{equation}

We first choose $\tau_1=\tau_2=0.1$ and hence $\mathbf{Hypothesis \  A1-\ A2}$ hold and $\varrho<0$. Thus it follows from \ref{thm5.1} that the equation admits a random attractor. Moreover,  $\frac{\varrho}{2}e^{[\frac{\varrho}{2}+L_v  (hL_f+bL_{\tilde{g}})]}+\tau-1$ $<$ $0$ also holds and it follows from Theorem \ref{Theorem 6} the attractor is a unique exponentially attractive  tempered random fixed point, i.e. the trivial fixed point $(0,0)$, which is shown in Fig 1.
In Fig 1a, the initial condition is  chosen to be $\phi(t)=(0.1, 0.2)$ for $t\in [-0.1,0]$ while in In Fig 1b, the initial condition is first chosen to be $\phi(t)=(10, 20)$ for $t\in [-0.1,0]$. From which we can see no matter how large the initial condition is, the state will tend to $(0,0)$.
\begin{figure}[htb]
\centering
\subfigure[]{
\includegraphics[width=0.45\textwidth]{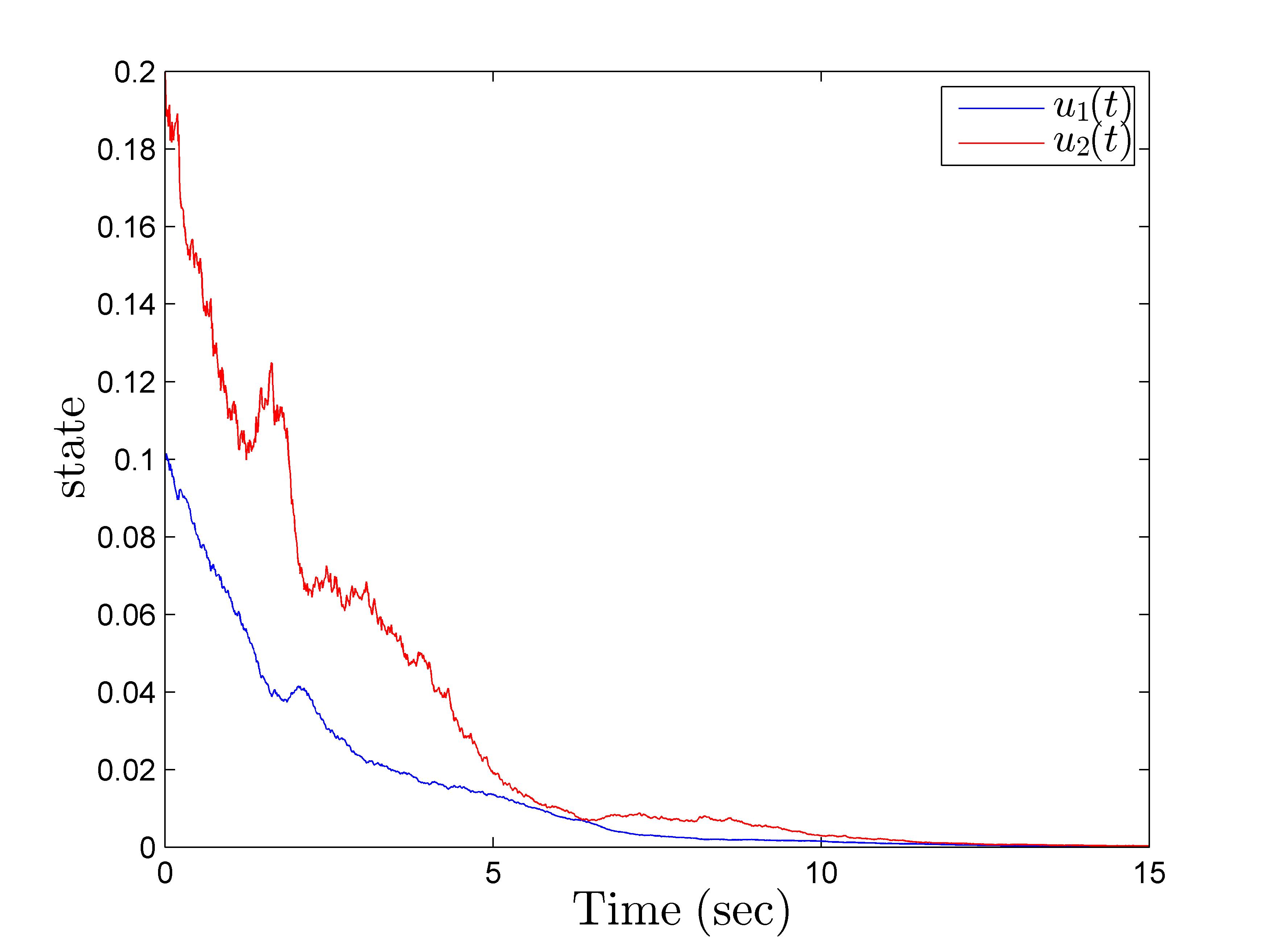}}
\subfigure[]{
\includegraphics[width=0.45\textwidth]{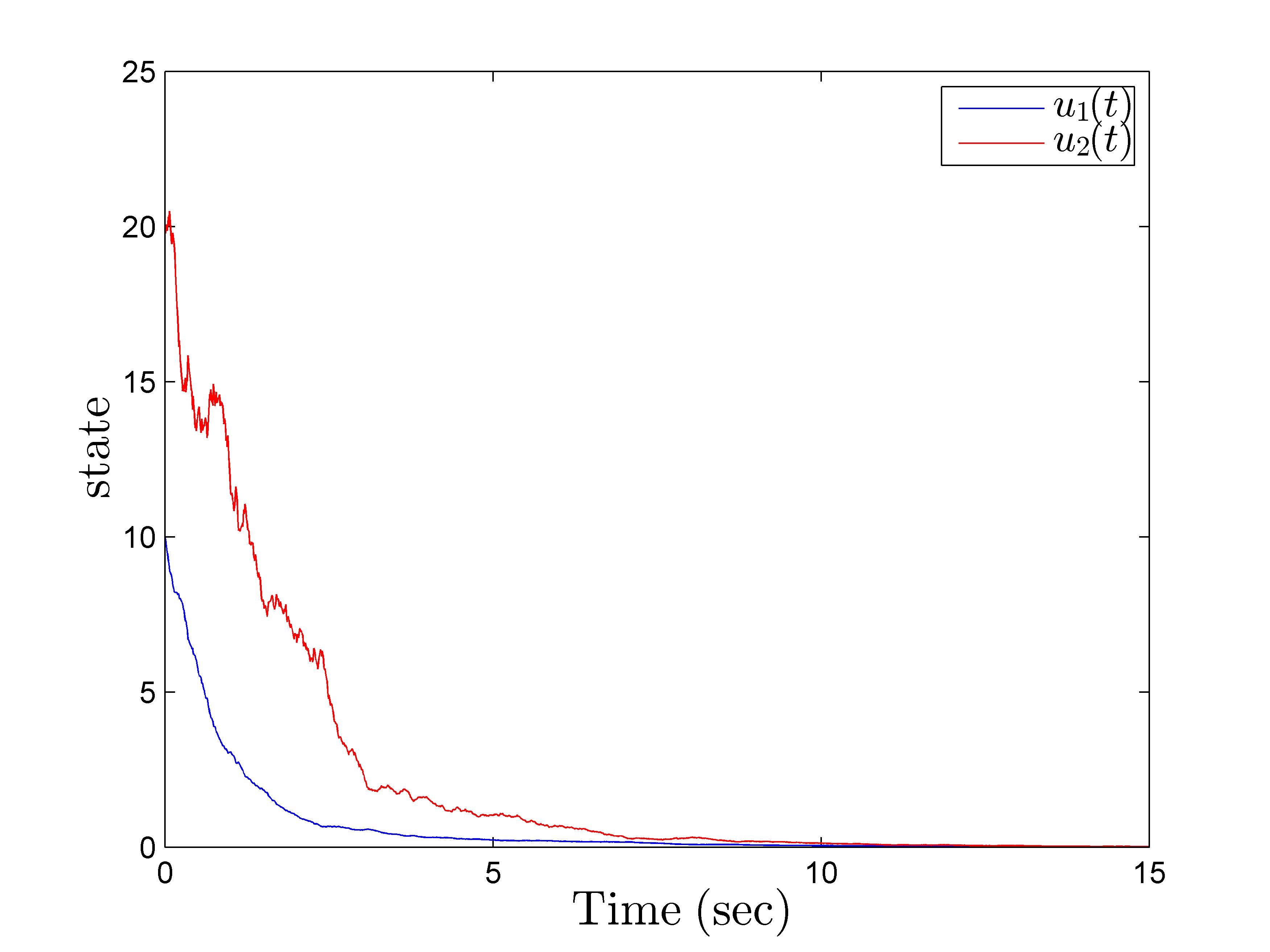}}
\caption{Solutions of (\ref{7.1}) with $\tau_1=\tau_2=0.1$  (a) initial condition $\phi(t)=(0.1, 0.2), t\in [-0.1,0]$, (b) initial condition $\phi(t)=(10, 20), t\in [-0.1,0]$.}
\end{figure}

\section{Summary and discussions}\label{Conclusions}
In this paper, we have obtained the existence and qualitative property of random attractors for \eqref{1.3}, which first required showing that the system generates a random dynamicla system. . We showed  under certain conditions  that the random attractor is a  globally exponentially attracting random stationary solution. From dynamical system theory, the conditions for the attractors to  be fixed point are so strong that could be hardly met in the real world applications. Indeed, from the dissipative system theory, if estimation  on the dimension of random attractors can  provide  researchers with useful information about the structure of the random attractor. Hence, the topological dimension for the random attractors of \eqref{1.3} will be investigated in a future paper. Furthermore, in order to obtain the  global complex dynamics and nonlocal analysis of the qualitative properties of the system, the existence and structure  of the associated invariant manifolds of the stationary solutions, the existence of connecting orbits (including the heteroclinic orbits or homoclinic orbits) are all of  significance and deserve  more attention.
\section{Acknowledgement}
This work was jointly supported by the National Natural Science Foundation
of China (62173139), China Postdoctoral Science Foundation(2019TQ0089),
Hunan Provincial Natural Science Foundation of China (2020JJ5344,
2019RS1033) the Science and Technology Innovation Program of Hunan
Province (2021RC4030), the Scientific Research Fund of Hunan Provincial
Education Department (20B353).
\section*{Acknowledgments}
The author thanks Dr Ding Kui for assistance of numerical simulations.

\end{document}